\documentclass{amsart}
\usepackage[latin1]{inputenc}
\usepackage{amsmath}
\usepackage{amsthm}
\usepackage{amsfonts}
\usepackage{amssymb}
\usepackage{enumerate}
\usepackage{mathrsfs}
\usepackage[all]{xy}
\usepackage{epic}
\usepackage{graphicx}
\usepackage[hypertex]{hyperref}
\usepackage{placeins}

\newcommand\bp{{\mathbb P}}

\newcommand\bc{{\mathbb C}}

\newcommand\ca{{\mathcal A}}

\newcommand\calF{{\mathcal F}}

\newtheorem{thm}{Theorem}
\newtheorem{conj}{Conjecture}

\newtheorem{lema}[thm]{Lemma}

\theoremstyle{remark}

\theoremstyle{definition}
\newtheorem{dfn}[thm]{Definition}

\newtheorem{ejm}[thm]{Example}

\numberwithin{equation}{section}
\numberwithin{figure}{section}

\title{Isotropic subspaces of Orlik-Solomon algebras}
\author{Miguel A. Marco-Buzun\'ariz}
\thanks{Research partially supported by the ERC Starting Grant project TGASS and
by Spanish Contract MTM2007-67908-C02-02. The author thanks to the Facultad de
Ciencias Matem\'aticas of the Universidad Complutense de Madrid and the
Universidad de Zaragoza for excellent working conditions.}
\begin{document}

\begin{abstract}
We give a combinatorial characterization of isotropic subspaces in the Orlik-Solomon algebra of a hyperplane arrangement in terms of decorations of its intersection lattice.
We then use this characterization to prove a result that relates these isotropic subspaces with linear systems supported on the arrangement, for arrangements with isolated non-normal crossings of a particular form.
\end{abstract}

\maketitle

\section{Introduction and first definitions}
The relationship between resonance varieties of line arrangements and pencils supported in them has been widely studied (see for instance \cite{falk-yuzvinsky-multinets}, \cite{libgober-yuzvinsky}, \cite{mio-graphscomb} or \cite{yuzvinsky-nets} among others). From all these works, we understand quite well the correspondence between these objects. But so far, no generalization to higher dimensions is known. One possible explanation for this is the fact that, in dimension two, components of the resonance varieties are the same than $2$-isotropic subspaces, but that is not true in higher dimensions. On the other hand, there are some similar results that relate isotropic subspaces of the cohomology ring with linear systems for compact Kaehler manifolds (see~\cite{catanese-moduli}). So it sounds natural to consider isotropic subspaces as the good generalization for the aforementioned results in line arrangements. In this paper, we give a first result in that direction: for a hyperplane arrangement whose non normal crossings are isolated points of the simplest type, then the isotropic subspaces of its Orlik-Solomon algebra correspond to linear systems supported in the arrangement. In order to do so, first we give a combinatorial criterion for a subspace to be isotropic (the so-called \emph{flag condition}), which is valid for any arrangement; and then we use this condition to prove the statement by induction using the extra conditions on the combinatorics of the arrangement. 

We conjecture that the following result holds for every arrangement:
\begin{conj}
\label{conjgen}
Let $A$ be a hyperplane arrangement in $\bc\bp^n$, and $V$ an $n$-isotropic subspace of its OS-algebra of dimension $d\geq n$ not contained in any coordinate hyperplane. Then the arrangement is a union of $d+1$ fibres of a linear system of hypersurfaces of dimension $n-1$.
\end{conj}
The techniques used in this paper don't seem to be powerful enough to prove
it. On the other hand, there is a class of arrangements for which we have hope
that this approach could be fruitful: the arrangements with isolated
non-normal crossings or INNC's, introduced by Libgober in~\cite{libgober-innc}.
In this setting, the previous conjecture can be stated in a much more precise
form:

\begin{conj}
Let $A=\{H_1,\ldots,H_l\}$ be a hyperplane arrangement in $\bc\bp^n$ with
isolated non normal crossings (INNC), and $V$ an $n$-isotropic subspace of its
OS-algebra of dimension $d\geq n$ not contained in any coordinate hyperplane.
Let $H_i$ be defined by the linear form $\alpha_i$. Then there exists a
partition $\Pi=\Pi_1,\ldots,\Pi_{d+1}$ of the set of hyperplanes of $A$, and a
choice of integer weights $w_i>0$ such that the equations 
\[
\left\{\prod_{H_i\in \Pi_j} \alpha_i^{w_i}\mid j=1,\ldots,{n+1}\right\}
\]
are linearly dependent.
\end{conj}

This second conjecture resembles the results known for line arrangements; and in fact the structure given by the partition $\Pi$ and the weights $w_i$ would be a direct generalization of the concept of \emph{multinet} or \emph{combinatorial pencil} presented in \cite{falk-yuzvinsky-multinets} and \cite{mio-graphscomb} respectively. From this point of view, the result presented in this paper could be seen as a generalization of the results known for \emph{nets} in line arrangements (as studied in \cite{yuzvinsky-nets}).

 Let us give precise definitions of the objects involved, and precise formulations of the statements. Most of the following definitions are classical in the theory of hyperplane arrangements (see~\cite{orlik-terao}).

Given a hyperplane arrangement $\ca=\{H_1,\ldots,H_l\}$ in $\bc\bp^n$, the cohomology ring of its complement is referred to as the Orlik-Solomon algebra of the arrangement. It is known to be isomorphic to the quotient $E/K$, where $E$ is the exterior algebra of $A^1:=\bc^l$ and $K$ is the ideal generated by
\[
\{\partial(x_{i_1}\wedge \cdots \wedge x_{i_m}):=\sum_{j=1}^m (-1)^j
x_{i_1}\wedge\cdots\wedge x_{i_{j-1}}\wedge x_{i_{j+1}}\wedge \cdots \wedge
x_{i_m}\mid H_{i_1},\cdots,H_{i_m} \text{intersect non generically}\}\]
and\[
\{x_{i_1}\wedge \cdots \wedge x_{i_m}\mid H_{i_1}\cap\cdots\cap H_{i_m}=\emptyset\},
\]
where $x_1,\cdots,x_l$ are the canonical generators of $E$ (one corresponding to each hyperplane).

This algebra is generated by $A^1$ and graded. We will denote by $A^i$ the homogeneous part of degree $i$.

Given an ordering in the hyperplanes $H_1,\ldots,H_l$ of the arrangement, the OS algebra has a particular basis (as a vectorspace) given by the set of monomials $x_{i_1}\wedge\cdots \wedge x_{i_k}$ such that:
\begin{itemize}
\item $k\leq n$
\item $i_1<\cdots<i_k$
\item the linear forms $\alpha_1\cdots,\alpha_k$ that define the hyperplanes $H_{i_1},\cdots ,H_{i_k}$ are linearly independent.
\item $H_{i_j}$ is the first hyperplane that contains $H_{i_j}\cap H_{i_{j+1}}\cap\cdots\cap H_{i_k}$ for all $1\leq j\leq k$.
\end{itemize}
This basis is called the \emph{Non Broken Circuit basis} (or NBC basis). In
particular, the elements of degree $j$ of this basis form a basis of $A^j$.

\begin{dfn}
Let $V$ be a linear subspace of $A^1$. We will say that $V$ is $k$-isotropic if $\bigwedge^k V=0$; that is, if the product of any $k$ vectors of $V$ vanishes.
\end{dfn}

It is easy to see that a $k$-isotropic subspace is also $k+1$ isotropic; and
that every subspace of dimension $k$ is $k+1$ isotropic. In this paper we will
be interested in $n$-isotropic subspaces of dimension at least $n$.

\section{Descendent well-fitted flags and the NBC basis}
We will consider a hyperplane arrangement in $\bc\bp^n$, and a fixed ordering in its hyperplanes. Let $A$ be its Orlik-Solomon Algebra. Consider $A^n$ its homogeneous part of degree $n$.

In the following we will consider flags embedded in the intersection lattice. That is, sequences of subspaces obtained by intersecting some hyperplanes of the arrangement.

\begin{dfn}
 A flag $F=(S_n\supset\cdots\supset S_1)$ is said to be \emph{descendent} if $j_{F,i+1}:=min\{j\mid S_{i+1}\subseteq H_j\}>j_{F,i}$ for any $i$.
\end{dfn}

This definition can be interpreted as follows. First we label each element of
the intersection lattice by the indices of the hyperplanes that contain it; and
then we see the flag as a walk in the intersection lattice that starts at a
hyperplane and ends at a point. We have then an increasing sequence of labels.
The condition of being descendent means that at each step we are adding some
index to the label that is smaller than all the previous ones.

\begin{ejm}
\label{ejemplo1}
 Consider a central arrangement of nine hyperplanes $H_1,\ldots,H_9$ in $\bc\bp^4$, whose homogeneous coordinates are given by the entries of the following system of equations:
\[
\left(x_1,x_2,x_3,x_4,x_5\right)\cdot
 \left(\begin{matrix}
        1 & 0 & 0 & 0 & 1 & 0 & 0 & 1 & 0\\
	0 & 1 & 0 & 0 & 1 & 1 & 0 & 0 & 0\\
	0 & 0 & 1 & 0 & 0 & 1 & 1 & 0 & 0\\
	0 & 0 & 0 & 1 & 0 & 0 & 1 & 1 & 0\\
	0 & 0 & 0 & 0 & 0 & 0 & 0 & 0 & 1\\
       \end{matrix}
\right) =\left(0\right).
\]

That is, there are eight hyperplanes that intersect in a point, and a ninth hyperplane that intersects generically.
We have the following non-generic intersections:
\begin{itemize}
 \item Codimension $1$ intersections: $H_1\cap H_2\cap H_5$, $H_2\cap H_3 \cap H_ 6$, $H_3\cap H_4\cap H_7$ and $H_1\cap H_4\cap H_8$.
\item Codimension $2$ intersections: $H_1\cap H_2\cap H_3 \cap H_5 \cap H_6$, $H_1 \cap H_2 \cap H_5 \cap H_7$, $H_3\cap H_4 \cap H_5 \cap H_7$, $H_2\cap H_3 \cap H_4 \cap H_6 \cap H_7$, $H_1 \cap H_2 \cap H_4 \cap H_5 \cap H_8$, $H_2 \cap H_3 \cap H_6 \cap H_8$, $H_1 \cap H_4 \cap H_6 \cap H_8$, $H_1\cap H_3 \cap H_4 \cap H_7\cap H_8$ and $H_5\cap H_6 \cap H_7 \cap H_8$.
\item One point (codimension $3$) that is the intersection of the first eight hyperplanes.
\end{itemize}

In this case, the flag 
\[
 (H_7\supset (H_3\cap H_4\cap H_7)\supset (H_2\cap H_3\cap H_4 \cap H_6 \cap
H_7)\supset (H_1\cap H_2\cap \cdots \cap H_8))
\]
is descendent, since the first index of each step is smaller than the previous one. On the other hand, the flag
\[
  (H_7 \supset (H_2\cap H_7)\supset (H_2\cap H_3\cap H_4 \cap H_6 \cap
H_7)\supset (H_1\cap H_2\cap \cdots \cap H_8))
\]
is not, because the first index of $H_2\cap H_7$ is the same than the one of $H_2\cap H_3\cap H_4 \cap H_6 \cap H_7$.
\end{ejm}

Note that, for each descendent flag $F=(S_n\supset \cdots \supset S_1)$, the
product $x_{j_{F,1}}\wedge \cdots \wedge x_{j_{F,n}}$ is an element of the Non
Broken Circuit basis. And vice-versa, if a monomial $x_{i_1}\wedge\cdots\wedge x_{i_n}$ is
an element of the NBC basis then the flag $F=(H_{i_n}\supset (H_{i_{n-1}}\cap
H_{i_n})\supset \cdots \supset (H_{i_1}\cap\cdots \cap H_{i_n}))$ is descendent
and $j_{F,k}$ is precisely $i_{k}$.

\begin{dfn}
 Let $L=(H_{i_1},\cdots,H_{i_n})$ a tuple of hyperplanes defined by linearly
independent $1$-forms, with $i_1<\ldots<i_n$. Let
$F=(S_n\supseteq\cdots\supseteq S_1)$ be a flag. We will say that $F$ is
\emph{well fitted} with respect to the tuple $L$ if there exists a permutation
$\sigma$ of the set $\{1,\ldots,n\}$ such that $S_k=\bigcap_{j=k}^n
H_{i_{\sigma(j)}}$ for each $k$. In such case, we define the relative
signature of $F$ with respect to the tuple $L$ as the signature of $\sigma$; it will be denoted by $sig_L(F)$.

The set of all descendent flags that are well fitted with respect to $L$ will be denoted by $Desc(L)$
\end{dfn}

The concept of well fitted means that the flag lies inside the intersection lattice of the subarrangement formed by the hyperplanes $H_{i_1},\cdots,H_{i_n}$.

\begin{lema}
\label{lemacoeficients}
 Let $L=(H_{i_1},\ldots,H_{i_n})$ be a linearly independent tuple. The expression of the corresponding element of the OS algebra in terms of the NBC basis is given by
\[
x_{i_1}\wedge\cdots\wedge
 x_{i_n}=(-1)^{{n}\choose{2}}
\sum_{F\in Desc(L)}sig_L(F)x_{j_{F,1}}\wedge \cdots \wedge x_{j_{F,n}}
\]	 
\end{lema}

\begin{proof}
Consider the following process of Gaussian elimination. Let $x_{i_1}\wedge\cdots \wedge x_{i_n}$ be an element of $A^n$ with $i_1 < i_2< \cdots < i_n$. Assume that $H_{i_1},\ldots,H_{i_n}$ form an independent set (otherwise, the wedge product would be directly zero). We will say that the monomial $x_{i_1}\wedge\cdots \wedge x_{i_n}$ has \emph{descendency failing tail} $m$ if $H_{i_k}$ is the first hyperplane that contains $H_{i_k}\cap H_{i_{k+1}}\cap\cdots \cap H_{i_n}$ for $k=1,\ldots,m-1$, but $H_{i_m'}$ is the first hyperplane that contains $H_{i_m}\cap\cdots \cap H_{i_n}$ being $i_m'< i_m$ (that is, we choose the first index in which the corresponding hyperplane can be substituted by another with lower index maintaining the intersection with the following ones). Note that $x_{i_1}\wedge\cdots \wedge x_{i_n}$ is an element of the NBC basis if and only if its descendency failing tail is not defined; in that case, we will say that it is $n+1$. Note also that the descendency failing tail cannot be equal to $n$, since the hyperplane $H_{i_n}$ is not contained in any other hyperplane of the arrangement.

Suppose now that the descendency failing tail of $x_{i_1}\wedge\cdots \wedge
x_{i_n}$ is $m$. This means that $\{H_{i_m'},H_{i_m},\ldots,H_{i_n}\}$ is a
dependent set, and hence we can use the relation
\[
\partial( x_{i_m'}\wedge x_{i_m}\wedge x_{i_{m+1}}\wedge \cdots\wedge x_{i_n})=0
\]

to express $x_{i_1}\wedge\cdots \wedge x_{i_n}$ as a sum of other monomials with higher descendency failing tail.
In particular, we have that
\[
\begin{array}{rcl}
x_{i_1}\wedge\cdots \wedge x_{i_n}&=&x_{i_1}\wedge\cdots \wedge x_{i_{m-1}}\wedge x_{i_m'}\wedge x_{i_{m+1}}\wedge \cdots\wedge x_{i_n}-\\

& & -x_{i_1}\wedge\cdots \wedge x_{i_{m-1}}\wedge x_{i_m'}\wedge x_{i_{m}}\wedge x_{i_{m+2}}\wedge \cdots\wedge x_{i_n}+
\\
& & +\cdots +\\
& & +(-1)^{n-m}\cdot x_{i_1}\wedge\cdots \wedge x_{i_{m-1}}\wedge x_{i_m'}\wedge x_{i_{m}}\wedge x_{i_{m+1}}\wedge \cdots\wedge x_{i_{n-1}}.
\end{array}
\]

We can keep applying this process to the resulting monomials, and at the end we will have a linear combination of monomials with descendency failing tail equal to $n+1$, that is, the expression of $x_{i_1}\wedge\cdots \wedge x_{i_n}$ in terms of the NBC basis.

Note that, at each step, the element
\[
 \frac{1}{x_{i_l}}(x_{i_1}\wedge\cdots \wedge x_{i_{m-1}}\wedge x_{i_m'}\wedge x_{i_{m}}\wedge \cdots\wedge \cdots\wedge x_{i_n})
\]
(where, $\frac{1}{x_{i_l}} $ represents the fact that the term is removed from the wedge product) corresponds to the flag 
\[
 (H_{i_n}\supset (H_{i_{n-1}}\cap H_{i_n})\supset\cdots\supset  (H_{i_1}\cap\cdots \cap H_{i_{m-1}}\cap H_{i_m'}\cap H_{i_{m}}\cap \cdots\cap {H_{i_{l-1}}} \cap{H_{i_{l+1}}} \cap\cdots\cap H_{i_n})),
\]
which is well fitted in the tuple $H_{i_1},\ldots,H_{i_n}$ with the permutation $(i_m,i_{m+1},\ldots,i_{l})$. Indeed, just consider the fact that $H_{i_m'}\cap H_{i_m}\cap\cdots\cap H_{i_{l-1}}\cap H_{i_{l+1}}\cap\ldots\cap H_{i_n}$ is the same as $H_{i_l}\cap H_{i_m}\cap\cdots\cap H_{i_{l-1}}\cap H_{i_{l+1}}\cap\ldots\cap H_{i_n}$.

That is, each term of each substitution process corresponds to a flag that is well fitted with respect to the original term. Since the descendency failing tail always increases, at the end we obtain elements whose corresponding flag is descendent; and by the previous argument, these flags should also be well fitted with respect to the original term (with the permutation obtained by composing the cycles obtained at each step).

To see that we obtain all the terms corresponding to well fitted descendent flags, take one such flag $F=(S_n\supset\cdots\supset S_1)$, and let $H_{j_{F,i}}$ be the first hyperplane that contains $S_i$. Assume that the flag is well fitted with respect to the term $x_{i_1}\wedge\cdots \wedge x_{i_n}$ with the permutation $\sigma$. In the previous process of substitution, we may obtain the term $x_{j_{F,1}}\wedge \cdots\wedge x_{j_{F,n}}$ as follows:
\[
\begin{array}{rcl}
 x_{i_1}\wedge\cdots\wedge x_{i_n}&=&\pm \frac{1}{x_{i_\sigma(1)}}(x_{j_{F,1}}\wedge x_{i_1}\wedge\cdots\wedge x_{i_n})\pm\cdots=\\
&=&\pm \frac{1}{x_{i_\sigma(2)}}\frac{1}{x_{i_\sigma(1)}}(x_{j_{F,1}}\wedge x_{j_{F,2}}\wedge x_{i_1}\wedge\cdots\wedge x_{i_n})\pm\cdots=\\
&=&\cdots\\
&=&\pm \frac{1}{x_{i_\sigma(n)}}\cdots\frac{1}{x_{i_\sigma(1)}}(x_{j_{F,1}}\wedge \cdots\wedge x_{j_{F,n}}\wedge x_{i_1}\wedge\cdots\wedge x_{i_n})\pm\cdots=\\
&=&x_{j_{F,1}}\wedge \cdots\wedge x_{j_{F,n}}\pm\cdots
\end{array}
\]

The previous expression shows that each monomial in the final expression is labeled by some permutation $\sigma$, and that the path followed to get to it keeps track of the permutation. Furthermore, the permutation $\sigma$ determines the process to get the monomial. From all those possible permutations, only the ones that correspond to a descendent flag will appear. That is, the substitution process has the structure of a tree isomorphic to some subtree of the permutations of $n$ elements.

 This means that each one of the final elements of the NBC basis appears only once; that is, the only possible coefficients in the final expression are $\pm1$.

Let's see now that the sign of these coefficients actually corresponds to the signature of the corresponding well-fitted descendent flag.

At each step of the process

\[
\begin{array}{rcl}
 x_{i_1}\wedge\cdots\wedge x_{i_n}&=&\pm \frac{1}{x_{i_\sigma(1)}}(x_{j_{F,1}}\wedge x_{i_1}\wedge\cdots\wedge x_{i_n})\pm\cdots=\\
&=&\pm \frac{1}{x_{i_\sigma(2)}}\frac{1}{x_{i_\sigma(1)}}(x_{j_{F,1}}\wedge x_{j_{F,2}}\wedge x_{i_1}\wedge\cdots\wedge x_{i_n})\pm\cdots=\\
&=&\cdots\\
&=&\pm \frac{1}{x_{i_\sigma(n)}}\cdots\frac{1}{x_{i_\sigma(1)}}(x_{j_{F,1}}\wedge \cdots\wedge x_{j_{F,n}}\wedge x_{i_1}\wedge\cdots\wedge x_{i_n})\pm\cdots=\\
&=&x_{j_{F,1}}\wedge \cdots\wedge x_{j_{F,n}}\pm\cdots
\end{array}
\]
we are substituting $x_{i_{\sigma(k)}}$ by $x_{j_{F,k}}$. In this substitution, the sign of the new term will depend on the order in which $\sigma(k)$ appears among the remaining $\sigma(k),\sigma(k+1),\cdots,\sigma(n)$. If $\sigma(k)$ is in an even order (that is, if there are an odd amount of indexes $j<k$ such that $\sigma(j)<\sigma(k)$) the sign will change, and it will remain the same otherwise. At the final step, we will have a $+$ sign if the total amount of pairs $j<k$ such that $\sigma(j)<\sigma(k)$ is even, and a $-$ sign otherwise. Since the signature of the permutation $\sigma$ can be counted as the parity of the set $\{(i,j)\mid i<j, \sigma(i)>\sigma(j)\}$, and there is a total amount of ${n}\choose{2}$ possible pairs, we have the result.

%
%
%
%
\end{proof}

\begin{ejm}
 We will now show this process for the arrangement in Example~\ref{ejemplo1}, for the monomial $x_3\wedge x_5\wedge x_6\wedge x_8$.

In the first step, we choose the first hyperplane through $H_3\cap H_5\cap H_6\cap H_8$, which is $H_1$, so we have:
\[
 x_3\wedge x_5\wedge x_6\wedge x_8=x_1\wedge x_5\wedge x_6\wedge x_8-x_1\wedge x_3\wedge x_6\wedge x_8+x_1\wedge x_3\wedge x_5\wedge x_8+x_1\wedge x_3\wedge x_5\wedge x_6.
\]

Now, for each one of these summands, we continue with the same process for the last three indices:
\begin{itemize}
 \item the first hyperplane through $H_5\cap H_6\cap H_8$ is $H_5$, so we don't need to change anything at this step (if we do the substitution, we will obtain 
\[x_5\wedge x_6 \wedge x_8=x_5\wedge x_6 \wedge x_8-x_5\wedge x_5 \wedge x_8+x_5\wedge x_5 \wedge x_6=x_5\wedge x_6 \wedge x_8,\]
that is, we don't change anything).
\item the first hyperplane through $H_3\cap H_6\cap H_8$ is $H_2$, so we have
\[
 x_1\wedge x_3\wedge x_6\wedge x_8=x_1\wedge x_2\wedge x_6\wedge x_8-x_1\wedge x_2\wedge x_3\wedge x_8+x_1\wedge x_2\wedge x_3\wedge x_6.
\]
\item the first hyperplane through $H_3\cap H_5\cap H_8$ is $H_3$, so again we don't need to change anything.
\item the first hyperplane through $H_3\cap H_5\cap H_6$ is $H_1$, so we have
\[
 x_1\wedge x_3\wedge x_5\wedge x_6=0.
\]
\end{itemize}
Finally, we just need to do the substitution in the third index. The first hyperplane through $H_6\cap H_8$ is $H_6$, the first one through $H_3\cap H_8$ is $H_3$ and the first one through $H_5\cap H_8$ is $H_5$. This, together with the fact that $H_2\cap H_3\cap H_6$ is a non generic intersection (that is, $x_2\wedge x_3\wedge x_6=0$), gives us that
\[
 x_3\wedge x_5\wedge x_6\wedge x_8=x_1\wedge x_5\wedge x_6\wedge x_8-x_1\wedge x_2\wedge x_6\wedge x_8+x_1\wedge x_2\wedge x_3\wedge x_8+x_1\wedge x_3\wedge x_5\wedge x_8
\]
is the expression in terms of the NBC basis. These summands correspond to the permutations $3568$, $5368$, $5638$ and $6358$, which are precisely the ones that correspond to the well-fitted descendent flags. Note that the rest of the permutations do not appear because they would give place to zero terms, precisely because of the corresponding flags being non descendent.

\end{ejm}
\section{Characterizing the vanishing products}

In this section we will use the previous formula to characterize the vanishing products in the Orlik-Solomon algebra.

 Let $(w_i=a_1^i\cdot x_1+\cdots + a_k^i\cdot x_l)_{i=1,\ldots,n}$ be an $n$-tuple of vectors in $A^1$. Consider the matrix $M=(a_i^j)$ whose rows are the coefficients of these vectors in the canonical basis. Denote by $v_1,\ldots ,v_k$ the column vectors of this matrix.
\begin{lema}
\label{lemadet1}
If we express the product $w_1\wedge\cdots\wedge w_n$ in terms of the NBC basis, the coefficient in $x_{i_1}\wedge \cdots \wedge x_{i_n}$ will be equal to the determinant of the matrix whose column vectors are
\[
 \sum_{H_i\supseteq \bigcap_{j=1}^n H_{i_j}}v_i,\cdots , \sum_{H_i\supseteq \bigcap_{j={n-1}}^n H_{i_j}}v_i,v_{i_n}
\]

\end{lema}

\begin{proof}
Note that the determinant is equal to the one of the matrix whose vectors are
\[
 \sum_{\substack{H_i\supseteq \bigcap_{j=1}^n H_{i_j}\\ H_i\nsupseteq \bigcap_{j=2}^n H_{i_j}}}v_i,\cdots , \sum_{\substack{H_i\supseteq \bigcap_{j={n-1}}^n H_{i_j}\\ H_i\neq H_{i_n}}}v_i,v_{i_n}.
\] 
If we expand this expression as a sum of $n\times n$ minors of the matrix $M$, we obtain precisely
\[
 \sum \left| \begin{matrix}
              v_{\sigma(j_1)} & v_{\sigma(j_{2})} &\cdots & v_{\sigma(j_n)} 
             \end{matrix}
\right|
\]
where the sum ranks through all independent tuples $H_{j_1},\ldots,H_{j_n}$ such that the flag $H_{i_n}\supseteq H_{i_n}\cap H_{i_{n-1}}\supseteq\cdots\supseteq H_{i_n}\cap\cdots\cap H_{i_1}$ is well fitted with respect to them, and $\sigma$ represents the corresponding permutations.

If we consider the vectors $w_i$ in the free exterior algebra, the coefficients of their wedge product in the canonical basis of $E^n$ will be precisely the $n\times n$ minors of the matrix $M$. Applying Lemma~\ref{lemacoeficients} we get the result.
\end{proof}

Given the vectors $w_i$, consider the following decoration of the intersection lattice of the arrangement. Decorate the hyperplane $H_i$ with the vector $v_i$. Then decorate each flat $S$ with $v_S:=\sum_{H_i\supseteq S}v_i$.
\begin{lema}\label{flagcondition}
The wedge product $w_1\wedge\cdots\wedge w_n$ is zero if and only if, for every flag $S_1\supseteq\cdots\supseteq S_n$, the vectors $v_{S_1},\ldots,v_{S_n}$ span a subspace of dimension lower than $n$.
\end{lema}

\begin{proof}
 From Lemma~\ref{lemadet1}, it is clear that the wedge product is zero if and only if the condition holds for the descendent flags (which are the ones corresponding to elements of the NBC basis). Let's see now that if the condition holds for all descendent flags, it must also hold for any flag.

In order to do so, we will see that the determinant corresponding to a non-descendent flag can be expressed as a linear combination of determinants corresponding to descendent flags.

Consider a flag $F^0=(S_n\supseteq\cdots\supseteq S_1)$. We will say that the descendency of $F^0$ fails at $i$ if $j_{F^0,i}=j_{F^0,i-1}$. We will call \emph{descendency failing head} to the biggest $i$ for which this occurs.

Consider the set $\calF:=\{F'=S_n\supseteq\cdots \supseteq S_{i+1}\supseteq S_{i'}\supseteq S_{i-1}\supseteq\cdots  \supseteq S_1)\}$ of flags that coincide with $F$ at every position except, maybe at the $i$'th one (in particular, $F_0\in \calF$). It is easy to check that
\[
 \sum_{(S_n\supset\cdots  \supset S_1)\in\calF} \left| 
\begin{matrix}
v_{S_n} & \cdots & v_{S_1}
 \end{matrix}
\right|=0.
\]
and hence we can express the determinant corresponding to $F_0$ as a linear combination of determinants corresponding to flags that coincide with $F$ at all steps except at position $i$. Note that, among all flags in $\calF$, only $F^0$'s descendency fails at  $i$. The rest of them either have descendency failing head equal to $i+1$, or smaller than $i$. 

For the flags with descendency failing head equal to $i+1$ obtained in the previous step, we can apply the same argument, expressing their determinants as a linear combination of the determinants corresponding to flags with descendency index equal to $i+2$ or smaller than $i$ (their descendency cannot fail at step $i$ because at levels $i$ and $i-1$ they coincide with flags obtained in the first step).

Note that if $i=n$, this process of ``pushing up'' the descendency failing head will just obtain flags with descendency failing head smaller than $n$.

So we have seen that, with this substitution process, we can express our flag $F^0$ as a linear combination of flags with smaller descendency failing head. An easy induction argument completes the proof. 
\end{proof}
If a choice of column vectors satisfy the hypothesis of Lemma~\ref{flagcondition}, we will say that they satisfy the \emph{flag condition}.

Note that the hypothesis of the lemma are equivalent to say that the subspace spanned by $w_1,\ldots,w_n$ is $n$-isotropic.

As a direct consequence of this result, we can generalize it to higher dimensional isotropic spaces.
\begin{lema}
 Let $V$ be a subspace of $A^1$, and let $M$ be a matrix whose rows are the coefficients of a generating system of $V$ in terms of the canonical basis of $A^1$. If we decorate the intersection lattice with vectors $v_S$ as before, then $V$ is $k$-isotropic if and only if for every truncated flag of length $k$ $F=(S_k\supset \cdots \supset S_1)$ with $dim(S_i)=i+n-k$, the vectors $v_{S_k},\ldots,v_{S_1}$ span a space of dimension lower than $k$. 
\end{lema}
\begin{proof}
If $k<n$, we can consider the arrangement induced by intersection with a generic $k$-dimensional subspace. The OS algebra and the intersection lattice of this new arrangement will be obtained from the original ones by truncating at the $k$'th level. So, without loss of generality, we can assume that $k=n$.

In this case, note that $V$ is $n$-isotropic if and only if the product of any $n$ generators vanishes. Given a flag $F=(S_k\supset \cdots \supset S_1)$, consider the matrix whose columns are $v_{S_k},\ldots,v_{S_1}$. Then by Lemma~\ref{flagcondition}, $V$ is isotropic if and only if the $n\times n$ minors of all these matrices are zero; which is equivalent to ask that these matrices have rank smaller than $n$.
\end{proof}

\section{Linear systems embedded in the arrangement}
We will now use he previous results to prove a pencil-like theorem for a special case of arrangements.

\begin{thm}
 Let $\ca\subseteq \bc^n$ be a hyperplane arrangement, with INNC and such that each INNC is the intersection of $n+1$ hyperplanes in a single point. Then there exists an $n$-isotropic subspace $V$ of dimension $n$, not contained in any coordinate hyperplane if and only if there exist a partition $\Pi=\Pi_1,\ldots,\Pi_{n+1}$ of $\ca$ such that:
\begin{itemize}
 \item Each element of the partition contains the same number of hyperplanes
\item There exists $(a_1,\ldots,a_{n+1})\in{\bc^\star}^{n+1}$ such that 
\[
 a_1\cdot f_1+\cdots +a_{n+1} \cdot f_{n+1}=0,
\]
where $f_i$ is the product of the linear forms defining the hyperplanes in $\Pi_i$.

\end{itemize}

Moreover, the subspace $V$ is generated by $\{\sum_{H_j\in \Pi_i} x_i-\sum_{H_j\in \Pi_{n+1}} x_i\mid 1\leq i\leq n\}$.
\end{thm}

\begin{proof}
If the arrangement is indeed a union of $n+1$ fibers of a linear system it is easy to construct the matrix $M$ that satisfies the previous flag condition, just as in the case of line arrangements. So let's focus on the other implication.

 It is already known for line arrangements (it is a straightforward consequence of the proof of Theorem~3.11 in \cite{falk-yuzvinsky-multinets}); so we can proceed by induction over $n$.

Take one hyperplane $H_0\in\ca$, and consider the restriction arrangement. Is is clear that this restriction arrangement must also be INNC with only $n$ hyperplanes going through each non normal crossing. If we take the matrix $M$ corresponding to the subspace $V$, and reduce all columns modulo $v_0$, we obtain a new subspace $V'$ in the Orlik-Solomon algebra of the corresponding restriction arrangement. This subspace will be inside the coordinate hyperplanes corresponding to those $H_i$ such that $v_i$ is proportional to $v_0$. So, if we consider the restriction subarrangement $\ca '$ formed by $\{H_0\cap H_j\mid v_j\notin \bc\cdot v_0\}$ we obtain an arrangement in $\bc^{n-1}$ and a subspace $V'\subseteq A'$ (being $A'$ the Orlik-Solomon algebra of the subarrangement $\ca'$). It is straightforward to check that the flag conditions for $V$ implies the flag conditions for $V'$, and hence we can apply the induction hypothesis and conclude that there exists a partition $\Pi'=\Pi'_1,\ldots,\Pi'_n$ of $\ca'$, and a choice of values $(a_1^0,\ldots ,a_n^0)$ such that $a_1^0 \cdot \bar f_1+\cdots +a_n^0 \cdot\bar f_n=\bar 0$, where $f_j=\prod_{H_i\in\Pi'_j}\alpha_i$, $a_i^0\in \bc\setminus \{0\}$, and the bars denote that we are considering the restriction to $H_0$. This implies that $\alpha_0$ divides $a_1^0 \cdot  f_1+\cdots +a_n^0 \cdot f_n$  (being $\alpha_0$ the linear form that defines $H_0$).

Let us define the partition $\Pi=\Pi_1,\ldots,\Pi_{n+1}$ as follows: for each $1\leq i\leq n$, $\Pi_i=\Pi_i'$; and $\Pi_{n+1}=\{H_j\mid v_j \in \bc \cdot v_0\}$ (that is, we add to $\Pi'$ the subset of the erased hyperplanes).

This reasoning can be applied for each $H_j$ such that $v_j\in \bc\cdot v_0$. Lets see that the coefficients $a_i^j$ must coincide. Suppose that $v_1$ is proportional to $v_0$. The restriction $H_j\cap H_0$ is a hyperplane in $H_0$ that cannot go through any of the base points of the induced linear system, since in that case we would have $n+2$ hyperplanes intersecting. If $a_1^1 \cdot \bar f_1+\cdots +a_n^1 \cdot\bar f_n$ wouldn't be $\bar 0$, it would determine a hypersurface in $H_0$ that contains $H_1\cap H_0$ as irreducible component. Since $H_0\cap H_1$ is a hyperplane that does not go through any base point of the linear system, we would have a fibre of a linear system with an irreducible component that does not go through any of the base points. This would give a contradiction, and hence the vectors $(a_1^0,\ldots ,a_n^0)$ and $(a_1^1,\ldots ,a_n^1)$ must be proportional. Repeating this reasoning for all possible $H_j$ such that $v_j$ is proportional to $v_0$, we get that $\alpha_j$ must divide $a_1^0 \cdot f_1+\cdots +a_n^0 \cdot f_n$ for every $H_j\in\Pi_{n+1}$. Since they are coprime, the product of all those $\alpha_j$ divides $a_1^0 \cdot f_1+\cdots +a_n^0 \cdot f_n$. Now we will see that the degrees will coincide, and that would finish the proof.

First note that, in the case of INNC, the flag condition is equivalent to saying that at each point, the vectors corresponding to the hyperplanes that contain it either span a subspace of rank lower than $n$, or they add up to zero.

Take one hyperplane $H_i\in \Pi_i$ in one of each of the first $n-1$ elements of the partition $\Pi$. The intersection of all them is a line $l$. This line intersect each of the remaining hyperplanes in exactly one point. If we count the intersections with the hyperplanes of $\Pi_n$, we get that, in order for the flag condition to be satisfied, a hyperplane of $\Pi_{n+1}$ must also go through it. Since that would mean already $n+1$ hyperplanes going through that point, there cannot be any other hyperplane joining there, and we have a pairing between the hyperplanes in $\Pi_n$ and the hyperplanes in $\Pi_{n+1}$. That, together with the induction hypothesis, gives the result.

\end{proof}

Note that the proof does not use the hypothesis in all the points, just in some of them; and hence the theorem could be true for a wider class of arrangements. In order to apply the proof, we just need the points  $p$ where $\sum_{p\in H_i} v_i=0$ to be the intersection of exactly $n+1$ hyperplanes. This justifies the following example.

\begin{ejm}\label{ejemplosimple}
 Consider the arrangements in $\bc\bp^n$ of the form 
\[(x_1^k-x_2^k)\cdot (x_2^k-x_3^k)\cdot\cdots\cdot (x_{n}^k-x_{n+1}^k)\cdot(x_{n+1}^k-x_1^k).\]
 They can be decomposed as the union of curves of the form $(x_i^k-x_{i+1}^k)$, that decompose as a union of hyperplanes. Since the sum of all their defining equations is zero, these curves belong to a linear system of dimension $n-1$. The base points of this linear system is the set of points $\{[a_1:a_2:\cdots a_n:1\mid a_1^k=a_2^k=\cdots=a_n^k=1\}$; and exactly $n+1$ hyperplanes of the arrangement go through each of these points. Then, even though the hypothesis of the theorem do not hold in general (for instance, these arrangements are not INNC if $k>2$, and there will exist points with more than $n+1$ hyperplanes if $n>3$), the proof would work nevertheless. And the same would happen for any arrangement obtained by a small deformation of these arrangements that respects the intersections at the base points. Taking a generic such deformation (when it exists), we would obtain a hyperplane arrangement satisfying the hypothesis of the theorem.
\end{ejm}
 
Very few examples of hyperplane arrangements that are formed by union of $n+1$ or more fibres of a linear system of dimension $n-1$ are known to the author. The following example shows one case in which some hyperplanes belong to more than one fibre. In these cases, Conjecture~\ref{conjgen} also holds, but their nature is essentially different: the non-isolated non-normal crossings play an important role.

\begin{ejm}
 Consider the arrangements in $\bc\bp^n$ given by equations $x_1\cdots x_{n+1}\cdot(x_1^k-x_2^k)\cdot (x_2^k-x_3^k)\cdots (x_{n}^k-x_{n+1}^k)\cdot(x_{n+1}^k-x_1^k)$. We have that
\[
\sum_{i=1}^n x_1^k\cdots x_{i-1}^k\cdot x_{i+2}^k\cdots x_{n+1}^k\cdot (x_i^k-x_{i+1}^k)+
x_2^k\cdots x_n^k\cdot(x_{n+1}^k-x_1^k)=0,
\]
so the arrangement is a union of $n+1$ fibres of a linear system of dimension $n-1$. But in this case, we can find base components of positive dimension (just take the intersection of the coordinate hyperplanes corresponding to $x_i$ and $x_{i+1}$). These components are non-isolated non-normal crossings in our arrangements, and hence we cannot hope to deform the arrangement to obtain a new one with a simpler combinatorial structure, but maintaining the desired property, as we did in Example~\ref{ejemplosimple}.
\end{ejm}

\bibliographystyle{abbrv}
\bibliography{biblio}

\end{document}